\documentclass[a4paper]{article}
  
  \usepackage{latexsym}
  \usepackage{amsmath}
  \usepackage{amsfonts}
  \usepackage{cite}
  \usepackage{color}
  
  \usepackage[T1]{fontenc}
  \usepackage[cp1250]{inputenc}
  \usepackage{graphics}

  \newtheorem{theorem}{Theorem}[section]
  \newtheorem{definition}{Definition}[section]
  \newtheorem{proposition}[theorem]{Proposition}

  \newenvironment{proof}[1][Proof]{\par\noindent\textbf{#1.}}{}
  \newcommand{\R}{\mathbb{R}}
  \newcommand{\N}{\mathbb{N}}
  
  \newcommand{\E}{\mathcal{E}}
  \newcommand{\im}{{\rm im}}
  \newcommand{\codim}{{\rm codim}}
  
  \title{Bifurcation of equilibrium forms of an elastic rod
  on a two-parameter Winkler foundation}
  
  \author{Marek Izydorek, Joanna Janczewska, Nils Waterstraat\\
  \& Anita Zgorzelska}

\begin{document}
  
    \maketitle
    
    \begin{abstract}
      We consider two-parameter bifurcation of equilibrium states of an elastic rod
      on a deformable foundation. Our main theorem shows that bifurcation occurs
      if and only if the linearization of our problem has nontrivial solutions.
      In fact our proof, based on the concept of the Brouwer degree, gives more,
      namely that from each bifurcation point there branches off a con\-ti\-nu\-um
      of solutions.
    \end{abstract}
    
    \textbf{key words:} Bifurcation, buckling, Winkler foundation.
    
    
    \textbf{running head:} Bifurcation of equilibrium forms.
             
    
    \section{Introduction}
    
    Bifurcation theory is one of the most powerful tools in studying
    deformations of elastic beams, plates and shells.
    Numerous works have been devoted to the study of bifurcation
    in elasticity theory (see for instance \cite{ChL}, \cite{Red}
    and the references therein).
    
    A familiar example from beam theory is the problem of stability
    of an isotropic elastic rod lying on a deformable foundation
    which is being compressed by forces at the ends (see Fig. \ref{rys1}).    
    For small forces the rod maintains its shape, however,
    as the forces increase they reach a first critical value
    beyond which the rod may buckle.
    
    In this work, we consider mixed boundary conditions which are as follows.
    The beam is free at the left end, and so it may move as in figure \ref{rys2} below.
    However, we require the shear force at the left end to vanish.
    At the right end, we assume the beam to be simply supported. 
        
    \begin{figure}
      \includegraphics{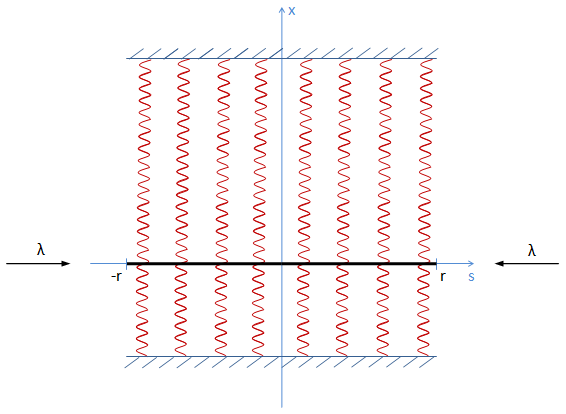}
      \caption{An elastic beam on an elastic foundation}\label{rys1}
    \end{figure}
    
    As we will show later, equilibrium forms of the rod
    under these boundary conditions satisfy the boundary value problem
    
    \begin{equation}\label{BE}
    \begin{cases}
      x^{(4)}+\alpha x^{''}+\beta x-f(x,x^{'},\ldots,x^{(4)})=0,
       & \textrm{in}\ [-r,r],\\
      x^{'}(-r)=x^{'''}(-r)=0, \\
      x(r)=x^{''}(r)=0,
    \end{cases}
    \end{equation}    
    where $\alpha$ is a parameter of the compressive force, $\beta$ is a parameter
    of the elastic foundation, and $f$ is a nonlinear term which we define
    in \eqref{nonlinearity} below.
    It follows from the definition of $f$ that for small forces the only solution
    of \eqref{BE} is the trivial one, i.e. $x_{0}(s)=0$, $s\in[-r,r]$,
    which corresponds to the straight rod in our bifurcation model.
    
    \begin{figure}
      \includegraphics{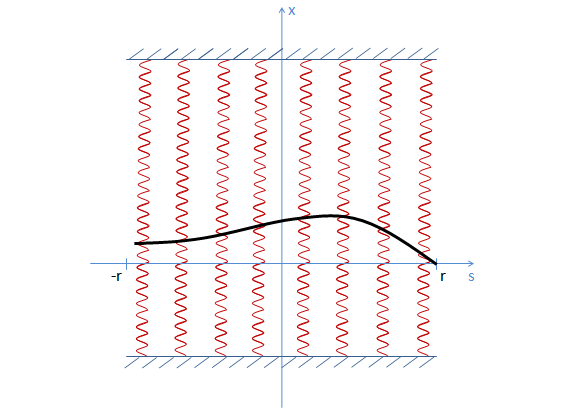}
      \caption{A buckling of an elastic beam}\label{rys2}
    \end{figure}
    
    However, as the forces increase the rod may buckle and it is desirable
    to know for which positive parameter values $(\alpha,\beta)$ this might happen.
    
    In order to answer this question, we associate with \eqref{BE}
    the linear boundary value problem
    
    \begin{equation}\label{LBE}
    \begin{cases}
      x^{(4)}+\alpha x^{''}+\beta x=0,
       & \textrm{in}\ [-r,r],\\
      x^{'}(-r)=x^{'''}(-r)=0, \\
      x(r)=x^{''}(r)=0
    \end{cases}
    \end{equation}
    and we denote by $N(\alpha,\beta)$ its space of solutions.
    
    The main theorem of this paper shows that a necessary and sufficient condition
    for bifurcation, and so for the possibility of a buckling of the rod,
    is that $\dim N(\alpha,\beta)\neq 0$.
    
    Let us point out that a similar model was investigated by A.\ Borisovich,
    Yu.\ Morozov and Cz.\ Szymczak in \cite{BMSz}, where the authors assumed
    that the rod is simply supported at both ends. They proved the existence
    of simple bifurcation points (meaning that $\dim N(\alpha,\beta)=1$)
    by applying a variational version of the Crandall-Rabinowitz theorem
    (compare Thm. \ref{CraRab} below).
    Later, in \cite{BoD}, A.\ Borisovich and J.\ Dymkowska showed
    a corresponding result under our boundary conditions,
    however, to the best of our knowledge
    the existence of multiple bifurcation points
    in the solution set of \eqref{BE}	is new. Note that here we prove even more,
    namely the existence of multiple branching points.
    
    Finally, let us mention that other models for buckling are described
    for example in \cite{AmP, BoB1, BoB2, BBO, BChT, CwR, Red}.
    
    Our paper is composed of three sections. In Section \ref{mathmod}
    we derive the equation of equilibrium forms of the rod and state our main theorem.
    Section \ref{proof} is devoted to the proof of this result.
    
    \subsubsection*{Acknowledgments.}
    
    Our research was supported by the Grant PPP-PL no.\ 57217076
    of the De\-uts\-cher Akademischer Austauschdienst - DAAD
    and the Ministry of Science and Higher Education of Poland - MNiSW.
    
    The authors wish to express their thanks
    to Professor Czes{\l}aw Szymczak from the Faculty of Ocean Engineering
    and Ship Technology of Gda\'{n}sk University of Technology
    for several helpful comments concerning the model.
    
    The authors are greatly indebted to Professor J\'{o}zef E.\ Sienkiewicz,
    the physicist from the Faculty of Applied Physics and Mathematics
    of Gda\'{n}sk University of Tech\-no\-lo\-gy, for pointing out
    a mistake in the formula for $E_2$ in \cite{BoD}.
    
    Our special thanks go to a student of mathematics at Gda\'{n}sk
    University of Technology, Aleksander Rogi\'{n}ski, for the preparation
    of pictures for the article.

    
    \section{Mathematical model}\label{mathmod}
    
    In this section we derive the equation \eqref{BE} of equilibrium forms
    of the rod by a variational approach along the lines of \cite{BoD}.
     The following formulas for $E_1$ and $E_3$ are as in \cite{BoD},
    but as a result of conversations with J.E.\ Sienkiewicz and Cz.\ Szymczak,
    the formula for $E_2$ has been improved.
    The authors of \cite{BoD} assumed that the rod
    under the action of the compressing force became longer,
    and so their assumption does not agree with experiments.
    Our refinement leads to a different nonlinear term in the equation \eqref{BE},
    however the system \eqref{LBE}, obtained by linearizing \eqref{BE},
    is not changed.
    
    Due to the fact that the work of A.\ Borisovich and J.\ Dymkowska contains
    a mistake, and moreover, it appeared only in Polish and in a limited number of copies,
    we do not restrict the discussion to explain the improvement,
    but for the convenience of the reader we provide a detailed exposition
    of the mathematical model.
     
    The total potential energy $E_t$ of the system composed of the rod
    and the foundation is equal to:
    
    \begin{displaymath}
      E_t=E_1-E_2+E_3,
    \end{displaymath}
    where
    
    \begin{itemize}
      \item $E_1$ is the energy of the compressed rod,
      \item $E_2$ is the work of the compressing force,
      \item $E_3$ is the energy of the Winkler foundation (i.e. of the springs).
    \end{itemize}
    
    The energy $E_1$ is given by
    
    \begin{displaymath}
      E_1(x)=\E I \int\limits^{r}_{-r}{\frac{\kappa^{2}(s)}{2}}ds,
    \end{displaymath}
    where
    
    \begin{displaymath}
      \kappa(s)=\frac{x''(s)}{\left(1+x'(s)^{2}\right)^{\frac{3}{2}}}
    \end{displaymath}
    is the curvature of the rod at a point $s\in[-r,r]$,
    $\E$ is Young's modulus and $I$ is the moment of inertia
    of the cross section of the rod.
    The second energy $E_2$ is defined as
    
    \begin{displaymath}
      E_2(x,\lambda)=
      \lambda\int\limits_{-r}^{r}\left(1-\sqrt{1-x'(s)^{2}}\right)ds,
    \end{displaymath}
    where
    
    \begin{displaymath}
      \int\limits_{-r}^{r}\left(1-\sqrt{1-x'(s)^{2}}\right)ds
    \end{displaymath}
    is the horizontal displacement of the left end of the rod
    and $\lambda>0$ is the value of the compressing force.
    Finally, the energy $E_3$ is defined by
    
    \begin{displaymath}
      E_3(x,\mu,\nu)=\int\limits^{r}_{-r}{U(x(s),\mu,\nu)ds},
    \end{displaymath}
		where
		
		\begin{displaymath}
		  U(x,\mu,\nu)=\frac{1}{2}\mu x^2-\frac{1}{4}\nu x^4 + o(x^4)
		\end{displaymath}
		is determined experimentally, and $\mu>0$ and $\nu>0$ are parameters
		of the elastic foundation.
		
		Expanding $(1+x)^{-3}$ and $\sqrt{1+x}$ as Maclaurin series,
		we get
		
		\begin{equation}\label{Mac1}
		  \frac{1}{\left(1+x\right)^3}=1-3x+6x^2-10x^3+o(x^3)
		\end{equation}
		and
		
		\begin{equation}\label{Mac2}
		  \sqrt{1+x}=1+\frac{x}{2}-\frac{x^2}{8}+\frac{x^3}{16}+o(x^3),
		\end{equation}
		respectively. If we omit the terms of order higher than $4$, we obtain
		
		\begin{displaymath}
		  E_1(x)\approx \E I \int_{-r}^{r}\left(\frac{1}{2}x''(s)^{2}
		  -\frac{3}{2}x'(s)^{2} x''(s)^{2}\right)ds
		\end{displaymath}
		and
		
		\begin{displaymath}
		  E_2(x,\lambda)\approx\lambda\int_{-r}^{r}\left(\frac{1}{2}x'(s)^{2}
		  +\frac{1}{8}x'(s)^{4}\right)ds.
		\end{displaymath}
		Hence the approximative formula for the total potential energy
		has the form
		
		\begin{align}\label{totalenergy}
		\begin{split}
		  E_t(x,\lambda,\mu,\nu)
		  \approx &\, \E I \int_{-r}^{r}\left(\frac{1}{2}x''(s)^{2}
		  -\frac{3}{2}x'(s)^{2} x''(s)^{2}\right)ds \\
		  & -\lambda\int_{-r}^{r}\left(\frac{1}{2}x'(s)^{2}
		  +\frac{1}{8}x'(s)^{4}\right)ds \\
		  & +\int_{-r}^{r}\left(\frac{1}{2}\mu x(s)^{2}
		  -\frac{1}{4}\nu x(s)^{4}\right)ds.
		\end{split}
		\end{align}
		
		We now define
		
		\begin{displaymath}
		  X=\{x\in C^4[-r,r]\colon x'(-r)=x'''(-r)=0,\ \ x(r)=x''(r)=0\}
		\end{displaymath}
		which is a Banach space with respect to the standard norm
		
		\begin{displaymath}
		  \|x\|_X=\sum^{4}_{k=0}{\max_{s\in [-r,r]}{|x^{(k)}(s)|}}.
		\end{displaymath}
		Note that the boundary conditions in the definition of $X$
		describe the behaviour of the rod at its ends (see Fig.\ \ref{rys2}).
		
		Setting
		
		\begin{displaymath}
		  \alpha=\frac{\lambda}{\E I},\ \ \beta=\frac{\mu}{\E I},\ \ \gamma=\frac{\nu}{\E I}.
		\end{displaymath}
		and dividing the formula \eqref{totalenergy} by $2r\E I$, we obtain
		a functional $E\colon X\times\R_{+}^{3}\to\R$ defined by
		
		\begin{align}\label{potential}
		\begin{split}
		  E(x,\alpha,\beta,\gamma)=
		  & \frac{1}{4r}\int_{-r}^{r}\left(x''(s)^{2}-3x'(s)^{2} x''(s)^{2}\right)ds\\
		  & -\frac{1}{4r}\int_{-r}^{r}\left(\alpha x'(s)^{2}+\frac{\alpha}{4}x'(s)^{4}\right)ds\\
		  & +\frac{1}{4r}\int_{-r}^{r}\left(\beta x(s)^{2}-\frac{\gamma}{2}x(s)^{4}\right)ds.
		\end{split}
		\end{align}
		In what follows we refer to $E$ as \textit{the energy functional},
		and we note for later reference that its derivative with respect
		to the space variable $x$ is
		
		\begin{align}\label{Euler}
		\begin{split}
		  E_{x}'(x,\alpha,\beta,\gamma)h=
		  & \frac{1}{2r}\int_{-r}^{r}\left(\beta x(s)-\gamma x(s)^{3}\right)h(s)ds\\
		  & -\frac{1}{2r}\int_{-r}^{r}\left(\alpha x'(s)+\frac{\alpha}{2}x'(s)^{3}
		  +3x'(s) x''(s)^{2}\right)h'(s)ds\\
		  & +\frac{1}{2r}\int_{-r}^{r}\left(x''(s)-3x'(s)^{2} x''(s)\right)h''(s)ds
		\end{split}
		\end{align}
		for all $x,h\in X$ and $\alpha,\beta,\gamma\in\R_{+}$.		
		Let us now denote by $Y$ the space $C[-r,r]$ with the standard norm
		
		\begin{displaymath}
		  \|y\|_{Y}=\max_{s\in[-r,r]}{|y(s)|},
		\end{displaymath}
		and let us consider the map $F\colon X\times\R_{+}^{3}\to Y$ defined by
		
		\begin{align}
		\begin{split}\label{map}
		  F(x,\alpha,\beta,\gamma)=\,
		  & x^{(4)}+\alpha x''+\beta x\\
		  & -\gamma x^3-3x''^{3}-12x'x''x'''\\
		  & -3x'^{2}\left(x^{(4)}-\frac{\alpha}{2}x''\right).
		\end{split}
		\end{align}
		
		If we set
		
		\begin{equation}\label{nonlinearity}
		  f(x,x',\ldots,x^{(4)})=\gamma x^3+3x''^{3}+12x'x''x'''
		  +3x'^{2}\left(x^{(4)}-\frac{\alpha}{2}x''\right)
		\end{equation}
		for each $x\in X$, then the operator equation
		
		\begin{equation}\label{OE}
		  F(x,\alpha,\beta,\gamma)=0
		\end{equation}
		is equivalent to our previously introduced boundary value problem \eqref{BE}.
		Clearly, the trivial function $x_0\equiv 0$ satisfies the equation \eqref{OE}
		for all values of parameters $\alpha$, $\beta$ and $\gamma$.
		We call the set $\Gamma\subset X\times\R_{+}^{3}$	given by
		
		\begin{displaymath}
		  \Gamma=\{(0,\alpha,\beta,\gamma)\colon \alpha,\beta,\gamma\in\R_{+}\}
		\end{displaymath}
		\textit{the trivial family} of solutions of the equation \eqref{OE}.
		Naturally, a solution of \eqref{OE} is said to be \textit{nontrivial}
		if it does not belong to $\Gamma$.
		
	 An interesting phenomenon is when there is a ''branching''
		of the equation \eqref{OE} in correspondence with some value
		of the multiparameter $(\alpha,\beta,\gamma)$. This is the object
		of bifurcation theory.
		
		\begin{definition}
		   A point $(0,\alpha_0,\beta_0,\gamma_0)\in\Gamma$ is called
		  a bifurcation point of \eqref{OE} if in every neighbourhood
		  of it in $X\times\R^{3}_{+}$ there is a nontrivial solution
		  of \eqref{OE}, in other words, $(0,\alpha_0,\beta_0,\gamma_0)$ belongs
		  to the closure in $X\times\R_{+}^{3}$ of the set of nontrivial
		  solutions of the equation \eqref{OE}.\\
		  In particular, a bifurcation point $(0,\alpha_0,\beta_0,\gamma_0)\in\Gamma$
		  of the equation \eqref{OE} is said to be a branching
		  point if there is a continuum (namely a closed connected set) of nontrivial
		  solutions of \eqref{OE} which contains $(0,\alpha_0,\beta_0,\gamma_0)$.
		\end{definition}
		
		Integrating by parts in \eqref{Euler}, we have
		
		\begin{align}
		\begin{split}\label{Lagrange}
		  E_{x}'(x,\alpha,\beta,\gamma)h=\,
		  & \frac{1}{2r}\int_{-r}^{r}\left(x^{(4)}(s)+\alpha x''(s)+\beta x(s)\right)h(s)ds\\
		  & -\frac{1}{2r}\int_{-r}^{r}\left(\gamma x(s)^{3}+3x''(s)^{3}+12x'(s)x''(s)x'''(s)\right)h(s)ds\\
		  & -\frac{1}{2r}\int_{-r}^{r}3x'(s)^{2}\left(x^{(4)}(s)-\frac{\alpha}{2}x''(s)\right)h(s)ds.
		\end{split}
		\end{align}
		
		If we denote by $\langle\cdot,\cdot\rangle$ the standard inner product
		in $L^{2}(-r,r)$, i.e.
		
		\begin{displaymath}
		  \langle g,h \rangle=\frac{1}{2r}\int_{-r}^{r}g(s)h(s)ds, \quad g,h\in L^{2}(-r,r),
		\end{displaymath}
		then
		
		\begin{equation}\label{vargrad}
		  E_{x}'(x,\alpha,\beta,\gamma)h=
		  \left\langle F(x,\alpha,\beta,\gamma),h\right\rangle
		\end{equation}
		for all $x,h\in X$ and $\alpha,\beta,\gamma\in\R_{+}$.
		Therefore, we call $F$ \textit{the variational gradient} of $E$,
		and we see from \eqref{vargrad} that solutions of \eqref{OE}
		are critical points of \eqref{potential}.
		
		Differentiating the map $F$ with respect to the space variable $x$
		at $x_0\equiv 0$ we get
		
		\begin{equation}\label{derivative}
		  F_{x}'(0,\alpha,\beta,\gamma)h=h^{(4)}+\alpha h''+\beta h
		\end{equation}
		for every $h\in X$ and $\alpha,\beta,\gamma\in\R_{+}$,
		and so
		
		\begin{displaymath}
		  N(\alpha,\beta)=\ker F_{x}'(0,\alpha,\beta,\gamma).
		\end{displaymath}
		
		We can now state the main result of this paper.
		
		\begin{theorem}\label{mainthm}
		  A point $(0,\alpha_0,\beta_0,\gamma_0)\in\Gamma$ is a branching point
		  of the equation \eqref{OE} if and only if $\dim N(\alpha_0,\beta_0)\neq 0$.
		\end{theorem}
		
		\hspace{10cm} $\Box$
		
		Our theorem extends Theorem 5.3.2 of \cite{BoD},
		which states that $\dim N(\alpha_0,\beta_0)\neq 0$ is a necessary condition
		for bifurcation in the solution set
		of the equation \eqref{OE} at a point $(0,\alpha_0,\beta_0,\gamma_0)$.
		
		It is worth pointing out that the theorem shows that the parameter $\gamma$
		has no influence on the occurrence of bifurcation.		
		
		\section{Proof of Theorem \ref{mainthm}}\label{proof}
		
		In order to prove Theorem \ref{mainthm}, we first discuss some properties
		of the nonlinear map $F$.
		
		\begin{proposition}\label{Fredholm}
		  For all values of parameters $\alpha,\beta,\gamma\in\R_{+}$
		  the linear operator $F'_{x}(0,\alpha,\beta,\gamma)\colon X\to Y$
		  is Fredholm of index zero.
		\end{proposition}
		
		\begin{proof}\,
		  The linear operator $A\colon C^4[-r,r]\rightarrow Y$, $Ah=h^{(4)}$
		  is surjective and its kernel consists of all polynomials of degree
		  at most $3$. Hence $A$ is Fredholm of index $4$.
		  As $X$ has codimension $4$ in $C^4[-r,r]$, the restriction of $A$
		  to $X$ is Fredholm of index $0$ (cf. \cite[Lemma XI.3.1]{Goldberg}).
		  Since the embeddings of $C^2[-r,r]$ and $C^4[-r,r]$ into $C[-r,r]$
		  are compact, it follows that $F'_{x}(0,\alpha,\beta,\gamma)$
		  is a compact perturbation of the restriction of $A$ to $X$
		  and so a Fredholm operator of index zero.
		\end{proof}
		
		\hspace{10cm} $\Box$
		
		The following proposition is an immediate consequence of the equality
		\eqref{vargrad}.
		
		\begin{proposition}\label{self-adjoint}
		  For all $\alpha,\beta,\gamma\in\R_{+}$
		  the map $F'_{x}(0,\alpha,\beta,\gamma)\colon X\to Y$ is self-adjoint
		  with respect to the inner product $\langle\cdot,\cdot\rangle$, i.e.
		  
		  \begin{displaymath}
		    \left\langle F'_{x}(0,\alpha,\beta,\gamma)h, g \right\rangle
		    =\left\langle h, F'_{x}(0,\alpha,\beta,\gamma)g \right\rangle
		  \end{displaymath}
		  for all $h,g\in X$.
		\end{proposition}
		
		\hspace{10cm} $\Box$
		
		We now denote by $Z$ the set of all points $(\alpha,\beta)\in\R_{+}^{2}$
		satisfying the inequality $4\beta\leq\alpha^{2}$.
		Let us consider in $Z$ the family of rays $l_m$ for $m\in\N$ given by
		
		\begin{displaymath}
		  \beta=-c_{m}\alpha-c_{m}^{2},
		\end{displaymath}
		where
		
		\begin{equation}\label{coefficients}
		  c_{m}=-\left(\frac{\pi}{r}\right)^{2}\left(\frac{2m-1}{4}\right)^{2}.
		\end{equation}
		
		\begin{theorem}[\cite{BoD}]\label{kernel}
		  For $(\alpha,\beta)\in\R_{+}^{2}$ one of the following three cases hold:
		  
		  \begin{enumerate}
			  \item[(i)] If the point $(\alpha,\beta)$ does not belong to any ray
			  $l_m$, then \[\dim N(\alpha,\beta)=0\] and the linear boundary value problem
			  \eqref{LBE} possesses only the trivial solution.
			  
			  \item[(ii)] If the point $(\alpha,\beta)$ belongs to one and only one ray
			  $l_m$, then \[\dim N(\alpha,\beta)=1\] and $N(\alpha,\beta)$
			  is generated by \[e_{m}(s)=2\cos{\sqrt{-c_m}(s+r)}.\]
			  
			  \item[(iii)] If the point $(\alpha,\beta)$ belongs to the intersection
			  of two rays $l_{m_1}$ and $l_{m_2}$ then \[\dim N(\alpha,\beta)=2\]
			  and the two linearly independent functions
			  \[e_{m_1}(s)=2\cos{\sqrt{-c_{m_1}}(s+r)}\] and \[e_{m_2}(s)=2\cos{\sqrt{-c_{m_2}}(s+r)}\]
			  are a basis of $N(\alpha,\beta)$.
      \end{enumerate}  
    \end{theorem}
    
    \hspace{10cm} $\Box$
    
    It follows from the implicit function theorem and Proposition \ref{Fredholm}
    that there is no bifurcation at points $(0,\alpha_0,\beta_0,\gamma_0)\in\Gamma$
    if $\dim N(\alpha_0,\beta_0)=0$. Hence Theorem \ref{kernel} shows that bifurcation
    can only occur at multiparameters $(\alpha,\beta,\gamma)$
    where $(\alpha,\beta)\in l_m$ for some $m$.

    Now, the rest of the proof of Theorem \ref{mainthm} splits into two cases,
    where we distinguish between simple and multiple branching points,
    i.e.\ whether the dimension of $N(\alpha,\beta)$ is $1$ or greater.
    
    \subsubsection*{Case 1.: Simple branching points.}
    
  In \cite{BoD}, A.\ Borisovich and J.\ Dymkowska proved the existence
    of bifurcation at a point $(0,\alpha_{0},\beta_{0},\gamma_{0})\in\Gamma$
    in the case $\dim N(\alpha_{0},\beta_{0})=1$ by applying the key function
    method due to Sapronov (see \cite{Sap}).
    
    For the convenience of the reader we present our own proof that is based
    on a variational version of the Crandall-Rabinowitz theorem
    on simple bifurcation points from \cite{Jan}, thus making our exposition
    self-contained.
    
    It will cause no confusion if we use the same letters $X, Y, \Gamma, F$
    and $E$ in the abstract result as in our issue.

    \begin{theorem}[see \cite{Jan}]\label{CraRab}
      Let $X$ and $Y$ be real Banach spaces that are continuously embedded
      in a real Hilbert space $H$ with inner product $\langle\cdot,\cdot\rangle$.
      
      Suppose that a $C^r$-smooth map $F\colon X\times\R\to Y$
      and a $C^{r+1}$-smooth functional $E\colon X\times\R\to\R$
      satisfy the conditions below:
      \begin{itemize}
        \item[$(C_1)$] $F(0,p)=0$ for all $p\in\R$,
        \item[$(C_2)$] $\dim\ker F'_{x}(0,p_0)=1$,
        \item[$(C_3)$] $\codim\,\im\, F'_{x}(0,p_0)=1$,
        \item[$(C_4)$] $E'_{x}(x,p)h=\langle F(x,p),h \rangle$
        for all $x,h\in X$ and $p\in\R$,
        \item[$(C_5)$] $E'''_{xxp}(0,p_0)(e,e,1)\neq 0$, where $e\in X$
        is such that $F'_{x}(0,p_0)e=0$, $\langle e,e\rangle=1$.
      \end{itemize}
      Then the set of solutions of the equation
      
      \begin{displaymath}
        F(x,p)=0
      \end{displaymath}
      in a small neighbourhood of $(0,p_0)$ is composed of two curves:
      $\Gamma$ and $\Lambda$, intersecting only at $(0,p_0)$,
      where $\Gamma$ is the trivial branch
      
      \begin{displaymath}
        \Gamma=\{(0,p)\in X\times\R\colon p\in\R\},
      \end{displaymath}
      and $\Lambda$ is a $C^{r+1}$-smooth curve that can be parametrized
      for some $\delta>0$ as
      
      \begin{displaymath}
        \Lambda=\{(x(t),p(t))\colon t\in(-\delta,\delta)\},
      \end{displaymath}
      where $x(0)=0$, $p(0)=p_{0}$ and $x'(0)=e$.
    \end{theorem}
    
    \hspace{10cm} $\Box$
    
    Combining \eqref{vargrad} with Proposition \ref{Fredholm},
    the proof of Theorem \ref{mainthm} in the first case will be completed
    by showing that at least one of the partial derivatives
    $E'''_{xx\alpha}(0,\alpha_{0},\beta_{0},\gamma_{0})(e_{m},e_{m},1)$
    or $E'''_{xx\beta}(0,\alpha_{0},\beta_{0},\gamma_{0})(e_{m},e_{m},1)$
    is not trivial, whe\-re $e_m$ is the function introduced
    in Theorem \ref{kernel}.
    
    An easy computation shows that
    
    \begin{equation}\label{alfa}
      E'''_{xx\alpha}(0,\alpha_{0},\beta_{0},\gamma_{0})(e_{m},e_{m},1)=
      -\langle e'_{m},e'_{m} \rangle=
      c_{m}<0
    \end{equation}
    and
    
    \begin{equation}\label{beta}
      E'''_{xx\beta}(0,\alpha_{0},\beta_{0},\gamma_{0})(e_{m},e_{m},1)=
      \langle e_{m},e_{m} \rangle=
      1>0.
    \end{equation}
    
    By Theorem \ref{CraRab}, $(0,\alpha_{0},\beta_{0},\gamma_{0})$ is
    a branching point of the equation \eqref{OE} both with respect to
    the parameter of compressive force $\alpha$ and with respect to
    the parameter of the elastic foundation $\beta$.
    Moreover, the solution set of \eqref{OE} in a small neighbourhood
    of $(0,\alpha_{0},\beta_{0},\gamma_{0})$ contains the trivial family
    $\Gamma$ and two $C^{\infty}$-smooth curves $\Lambda_{1}$,
    $\Lambda_{2}$ of the form
    
    \begin{displaymath}
      \Lambda_{1}=\{(x_{1}(t),\alpha(t))\colon |t|<\delta_{1}\}
      \subset X\times\R_{+}\times\{(\beta_{0},\gamma_{0})\},
    \end{displaymath}
    where $x_{1}(0)=0$, $\alpha(0)=\alpha_{0}$, $x'_{1}(0)=e_{m}$,
    and
    
    \begin{displaymath}
      \Lambda_{2}=\{(x_{2}(t),\beta(t))\colon |t|<\delta_{2}\}
      \subset X\times\R_{+}\times\{(\alpha_{0},\gamma_{0})\},
    \end{displaymath}
    where $x_{2}(0)=0$, $\beta(0)=\beta_{0}$, $x'_{2}(0)=e_{m}$.
    Hence $(0,\alpha_0,\beta_0,\gamma_0)$ is a branching point.
    
    \subsubsection*{Case 2.: Multiple branching points.}
    
    We now turn to multiple branching points.
    
    Here the method based on the Crandall-Rabinowitz theorem
    does not work anymore. In order to prove the existence of branching points
    also in this case, we will make a finite-dimensional reduction
    of Lyapunov-Schmidt type.
    
    Let $\alpha_0,\beta_0,\gamma_0\in\R_{+}$ be such that \[\dim N(\alpha_0,\beta_0)=2,\]
    and let $e_{m_1}$ and $e_{m_2}$ be the corresponding functions in Theorem \ref{kernel}.
    Since
    
    \begin{displaymath}
      (\alpha_0,\beta_0)\in
      \overline{\{(\alpha,\beta)\in\R_{+}^{2}\colon \dim N(\alpha,\beta)=1\}},
    \end{displaymath}
    we see that $(0,\alpha_0,\beta_0,\gamma_0)$ is a bifurcation point,
    and we shall now show that it is a branching point. 
    We define a map $G\colon X \times \R^{2} \times \R_{+} \times \R_{+} \to Y$ by
    
		\begin{displaymath}
   	  G(x,\xi,\alpha,\beta)=F(x,\alpha,\beta,\gamma_0)
   	  +\sum_{i=1}^{2}\left(\xi_{i}-\left\langle x, e_{m_i}\right\rangle\right)e_{m_i},
    \end{displaymath}
    where $x\in X$, $\xi=(\xi_{1},\xi_{2})\in\R^2$ and $\alpha,\beta\in\R_{+}$.
    
    It is easily seen that
    
    \begin{displaymath}
      G'_{x}(0,0,\alpha_0,\beta_0)h=F'_{x}(0,\alpha_0,\beta_0,\gamma_0)
      -\sum_{i=1}^{2}\left\langle h, e_{m_i}\right\rangle e_{m_i},
    \end{displaymath}
    where $h\in X$, is an isomorphism of $X$ onto $Y$.
    
    By the implicit function theorem there exist open subsets
    $U\subset X$ and	$S \subset \R^{2} \times \R_{+} \times \R_{+}$
    such that $0\in U$, $(0,\alpha_0,\beta_0)\in S$,
    and the set
    
    \begin{displaymath}
      \{(x,\xi,\alpha,\beta)\in U\times S \colon G(x,\xi,\alpha,\beta)=0\}
    \end{displaymath}
    is the graph of a smooth function $\tilde{x}\colon S\to U$ satisfying  
		$\tilde{x}(0,\alpha_0,\beta_0)=0$. Moreover, since $G(0,0,\alpha,\beta)=0$
		for all $\alpha,\beta\in\R_{+}$, it follows that $\tilde{x}(0,\alpha,\beta)=0$
		for all $(0,\alpha,\beta)\in S$.
		
		We now introduce a function $\varphi=(\varphi_1,\varphi_2)\colon S\to\R^{2}$ by
		
		\begin{equation}\label{var} 
		  \varphi_{i}(\xi,\alpha,\beta)=
		  \xi_{i}-\left\langle \tilde{x}(\xi,\alpha,\beta), e_{m_i} \right\rangle,\,\, i=1,2,
		\end{equation}
		and we note that $\varphi$ is smooth and $\varphi(0,\alpha,\beta)=0$
		for all $(0,\alpha,\beta)\in S$.
		
		\begin{theorem}[see \cite{Jan}]{\label{LapSch}}
		
		The point $(0,\alpha_0,\beta_0,\gamma_0)\in X \times \R_{+}^{3}$
		is a bifurcation point (a branching point) of \eqref{OE}
		if and only if the point $(0,\alpha_0,\beta_0)\in \R^{2} \times \R_{+} \times \R_{+}$
		is a bifurcation point (a branching point) of the equation
		
		\begin{equation}\label{RE}
      \varphi(\xi,\alpha,\beta)=0.  
    \end{equation}
    \end{theorem}
    
    \hspace{10cm} $\Box$
    
    The rest of the argument is based on the concept of topological degree
    due to Brouwer. To be more precise, we will apply a theorem of Krasnosielski,
    which we recall for the convenience of the reader.
    
    \begin{theorem}[see \cite{Jan}]\label{Kras}
      If $(0,\lambda_{0},\beta_{0})\in S$ is not a bifurcation point
      of equation \eqref{RE} then there exist open sets $V_{1}\subset\R^{2}$
      and $V_{2}\subset\R_{+}\times\R_{+}$ satisfying:
      
      \begin{itemize}
        \item[$(i)$] $(0,\alpha_{0},\beta_{0})\in V_{1}\times V_{2}\subset S$.
        \item[$(ii)$] For each open subset $V\subset V_{1}$ such that $0\in V$
        and for all $(\alpha,\beta),(\tilde\alpha,\tilde\beta)\in V_{2}$
        the mappings $\varphi(\cdot,\alpha,\beta)$
        and $\varphi(\cdot,\tilde\alpha,\tilde\beta)$ have no zeros
        on the boundary of $V$ and
        
        \begin{equation}\label{KrasAlter}
          \deg(\varphi(\cdot,\alpha,\beta),V,0)=
          \deg(\varphi(\cdot,\tilde\alpha,\tilde\beta),V,0).
        \end{equation}
      \end{itemize}
    \end{theorem}
    
    \hspace{10cm} $\Box$
    
    Here and subsequently, $\deg\left(\varphi(\cdot,\alpha,\beta),V,0\right)$
    stands for the Brouwer degree of the map $\varphi(\cdot,\alpha,\beta)$
    on the set $V$ with respect to $0$.
    
    We do not want to recapitulate degree theory here, however,
    let us point out the important fact that in our case
    for each $(\alpha,\beta)\in V_2$ there is a neighbourhood
    $V\subset V_{1}$ of $0$ such that 
    
    \begin{displaymath}
      \deg\left(\varphi(\cdot,\alpha,\beta),V,0\right)=
      \mbox{sgn}\det[\varphi'_{\xi}(0,\alpha,\beta)].
    \end{displaymath}
    
    We now proceed to show that $(0,\alpha_0,\beta_0)\in \R^{2} \times \R_{+} \times \R_{+}$
		is a branching point of \eqref{RE}.
		It is well-known from bifurcation theory and degree theory that it is sufficient
		to prove that the equality \eqref{KrasAlter} does not hold.
		 
		Differentiating
		
		\begin{displaymath}
      G(\tilde{x}(\xi,\alpha,\beta),\xi,\alpha,\beta)=0  
    \end{displaymath}
		with respect to $\xi$ we get
		
		\begin{align*}
		\begin{split}
      F'_{x}(\tilde{x}(\xi,\alpha,\beta),\alpha,\beta,\gamma_0)&
      \sum_{j=1}^{2}\frac{\partial\tilde{x}}{\partial\xi_{j}}(\xi,\alpha,\beta)t_{j}
      +\sum_{j=1}^{2}t_{j}e_{m_j}\\
      &-\sum_{i=1}^{2}\sum_{j=1}^{2}\left\langle
      \frac{\partial\tilde{x}}{\partial\xi_{j}}(\xi,\alpha,\beta)t_{j}, e_{m_i}
      \right\rangle e_{m_i}=0,
    \end{split}
    \end{align*}
		for all $t=(t_1,t_2)\in\R^{2}$. Hence
		
		\begin{align}
		\begin{split}\label{impliciteq}
      F'_{x}(0,\alpha,\beta,\gamma_0)&
      \sum_{j=1}^{2}\frac{\partial\tilde{x}}{\partial\xi_{j}}(0,\alpha,\beta)t_{j}
      +\sum_{j=1}^{2}t_{j}e_{m_j}\\
      &-\sum_{i=1}^{2}\sum_{j=1}^{2}\left\langle
      \frac{\partial\tilde{x}}{\partial\xi_{j}}(0,\alpha,\beta)t_{j}, e_{m_i}
      \right\rangle e_{m_i}=0,
    \end{split}
    \end{align}
    and combining \eqref{impliciteq} and \eqref{map} we have
    
    \begin{align}
    \begin{split}\label{impmap}
      \frac{d^{4}}{ds^{4}}
      \sum_{j=1}^{2}\frac{\partial\tilde{x}}{\partial\xi_{j}}(0,\alpha,\beta)t_{j}
      & +\alpha\frac{d^{2}}{ds^{2}}
      \sum_{j=1}^{2}\frac{\partial\tilde{x}}{\partial\xi_{j}}(0,\alpha,\beta)t_{j}
      +\beta\sum_{j=1}^{2}\frac{\partial\tilde{x}}{\partial\xi_{j}}(0,\alpha,\beta)t_{j}\\
      & +\sum_{j=1}^{2}t_{j}e_{m_j}
      -\sum_{i=1}^{2}\sum_{j=1}^{2}\left\langle
      \frac{\partial\tilde{x}}{\partial\xi_{j}}(0,\alpha,\beta)t_{j}, e_{m_i}
      \right\rangle e_{m_i}=0.
    \end{split}
    \end{align}
    If we now substitute into $t=(t_1,t_2)$ the vectors $(1,0)$ and $(0,1)$ subsequently,
    we obtain
    
    \begin{align}
    \begin{split}
      \frac{d^{4}}{ds^{4}}
      \frac{\partial\tilde{x}}{\partial\xi_{j}}(0,\alpha,\beta)
      & +\alpha\frac{d^{2}}{ds^{2}}
      \frac{\partial\tilde{x}}{\partial\xi_{j}}(0,\alpha,\beta)
      +\beta\frac{\partial\tilde{x}}{\partial\xi_{j}}(0,\alpha,\beta)\\
      & +e_{m_j}-\sum_{i=1}^{2}\left\langle
      \frac{\partial\tilde{x}}{\partial\xi_{j}}(0,\alpha,\beta), e_{m_i}
      \right\rangle e_{m_i}=0
    \end{split}
    \end{align}
    for $j=1,2$. Therefore
    
    \begin{align*}
    \begin{split}
      \left\langle\frac{d^{4}}{ds^{4}}
      \frac{\partial\tilde{x}}{\partial\xi_{j}}(0,\alpha,\beta),e_{m_k}\right\rangle
      & +\left\langle\alpha\frac{d^{2}}{ds^{2}}
      \frac{\partial\tilde{x}}{\partial\xi_{j}}(0,\alpha,\beta),e_{m_k}\right\rangle
      +\left\langle
      \beta\frac{\partial\tilde{x}}{\partial\xi_{j}}(0,\alpha,\beta),e_{m_k}
      \right\rangle\\
      & +\left\langle e_{m_j},e_{m_k}\right\rangle-\left\langle
      \frac{\partial\tilde{x}}{\partial\xi_{j}}(0,\alpha,\beta), e_{m_k}\right\rangle=0
    \end{split}
    \end{align*}
    for $j=1,2$ and $k=1,2$. Applying Proposition \ref{self-adjoint} 
    we see that
    
    \begin{equation}\label{adjoint}
      \left\langle\frac{\partial\tilde{x}}{\partial\xi_{j}}(0,\alpha,\beta),
      e_{m_k}^{(4)}+\alpha e''_{m_k}+\beta e_{m_k}-e_{m_k}\right\rangle=
      -\left\langle e_{m_j}, e_{m_k} \right\rangle
    \end{equation}
    for $j=1,2$ and $k=1,2$.
    Since $e''_{m_k}=c_{m_k}e_{m_k}$ and $e^{(4)}_{m_k}=c_{m_k}^{2}e_{m_k}$
    for $k=1,2$, we obtain
    
    \begin{displaymath}
	    \left\langle
	    \frac{\partial\tilde{x}}{\partial\xi_{j}}(0,\alpha,\beta),e_{m_k}
	    \right\rangle=
	    \left\{
	      \begin{array}{cl}
	        -\frac{1}{c_{m_k}^{2}+\alpha c_{m_k}+\beta-1} & \textrm{if $j=k$} \\
	        0 & \textrm{if $j\neq k$}
	      \end{array}
	    \right.,
	  \end{displaymath}
	  by \eqref{adjoint}. Now \eqref{var} yields
	  
	  \begin{displaymath}
		  \frac{\partial\varphi_k}{\partial\xi_{j}}(0,\alpha,\beta)=
		  \left\{
		    \begin{array}{cl}
        \frac{c_{m_k}^{2}+\alpha c_{m_k}+\beta}{c_{m_k}^{2}+\alpha c_{m_k}+\beta-1}
        & \mbox{if $j=k$} \\
        0 & \mbox{if $j\neq k$}
        \end{array}
      \right.
		\end{displaymath}
		and so
		
		\begin{equation}\label{matrix}
		  [\varphi'_{\xi}(0,\alpha,\beta)]=
		  \left[
		    \begin{array}{cc}
		      \frac{c_{m_1}^{2}+\alpha c_{m_1}+\beta}{c_{m_1}^{2}+\alpha c_{m_1}+\beta-1}
		      & 0 \\
		      0
		      & \frac{c_{m_2}^{2}+\alpha c_{m_2}+\beta}{c_{m_2}^{2}+\alpha c_{m_2}+\beta-1}
		    \end{array}
		  \right].
		\end{equation}
		Furthermore, it follows from Theorem \ref{kernel} that
		
		\begin{displaymath}
		  \beta_{0}=-c_{m_1}\alpha_{0}-c_{m_1}^{2}
		  \,\,\, \mbox{and} \,\,\, 
		  \beta_{0}=-c_{m_2}\alpha_{0}-c_{m_2}^{2},
		\end{displaymath}
		and so
		
		\begin{displaymath}
		  c_{m_1}^{2}=-c_{m_1}\alpha_{0}-\beta_{0}
		  \,\,\, \mbox{and} \,\,\, 
		  c_{m_2}^{2}=-c_{m_2}\alpha_{0}-\beta_{0}.
		\end{displaymath}
		Hence \eqref{matrix} now becomes
		
		\begin{displaymath}
		  [\varphi'_{\xi}(0,\alpha,\beta)]=
		  \left[
		    \begin{array}{cc}
		      \frac{(\alpha-\alpha_{0})c_{m_1}+\beta-\beta_{0}}
		      {(\alpha-\alpha_{0})c_{m_1}^{2}+\beta-\beta_{0}-1} & 0 \\
		      0 &  \frac{(\alpha-\alpha_{0})c_{m_2}+\beta-\beta_{0}}
		      {(\alpha-\alpha_{0})c_{m_2}^{2}+\beta-\beta_{0}-1}
		    \end{array}
		  \right],
		\end{displaymath}
		and in consequence,
		
		\begin{equation}\label{determinant}
		  \det[\varphi'_{\xi}(0,\alpha,\beta)]=
		  \frac{(\alpha-\alpha_{0})c_{m_1}+\beta-\beta_{0}}
		  {(\alpha-\alpha_0)c_{m_1}^{2}+\beta-\beta_{0}-1}
		  \cdot\frac{(\alpha-\alpha_{0})c_{m_2}+\beta-\beta_{0}}
		  {(\alpha-\alpha_{0})c_{m_2}^{2}+\beta-\beta_{0}-1}.
		\end{equation}
		
		Our aim is now to determine the sign of \eqref{determinant} 
		at points in a small neighbourhood of $(0,\alpha_{0},\beta_{0})$.
		
		We first note that there exists $\varepsilon>0$
		such that the denominator of \eqref{determinant} is positive for every
		$(\alpha,\beta)\in\left(\alpha_{0}-\varepsilon,\alpha_{0}+\varepsilon\right)
		\times\left(\beta_{0}-\varepsilon,\beta_{0}+\varepsilon\right)$.
		
		Let $n(\alpha,\beta)$ denote the numerator of \eqref{determinant},
		i.e.
		
		\begin{displaymath}
		  n(\alpha,\beta)=
		  \left((\alpha-\alpha_{0})c_{m_1}+\beta-\beta_{0}\right)
		  \cdot\left((\alpha-\alpha_{0})c_{m_2}+\beta-\beta_{0}\right).
		\end{displaymath}
		
		For $\alpha\neq\alpha_{0}$ we have
		
		\begin{displaymath}
		  n(\alpha,\beta)=
		  \left(\alpha-\alpha_{0}\right)^{2}
		  \left(c_{m_1}+\frac{\beta-\beta_{0}}{\alpha-\alpha_{0}}\right)
		  \left(c_{m_2}+\frac{\beta-\beta_{0}}{\alpha-\alpha_{0}}\right).
    \end{displaymath}
		We can assume without loss of generality that $m_{1}<m_{2}$.
		Then $c_{m_1}>c_{m_2}$ by \eqref{coefficients}, and we can check at once that
		
		\begin{displaymath}
		  \mbox{sign}\det[\varphi'_{\xi}(0,\alpha,\beta)]=
		  \left\{
		    \begin{array}{ll}
		      1 & \mbox{if $\frac{\beta-\beta_{0}}{\alpha-\alpha_{0}}
		      \in(-\infty,-c_{m_1})\cup(-c_{m_2},\infty)$}\\
		      -1 & \mbox{if $\frac{\beta-\beta_{0}}{\alpha-\alpha_{0}}
		      \in(-c_{m_1},-c_{m_2})$}
		    \end{array}
		  \right..
		\end{displaymath}
		Now let us suppose, contrary to our claim, that $(0,\alpha_{0},\beta_{0})$
		is not a bifurcation point of the equation \eqref{RE}.
		Let $V_{1}\subset\R^{2}$ and $V_{2}\subset\R_{+}\times\R_{+}$
		be the open sets as in Theorem \ref{Kras}.
		Clearly, there are $(\alpha,\beta)$	and $(\tilde\alpha,\tilde\beta)$
		in $V_{2}\cap(\alpha_{0}-\varepsilon,\alpha_{0}+\varepsilon)
		\times(\beta_{0}-\varepsilon,\beta_{0}+\varepsilon)$ such that
		
		\begin{displaymath}
		  \frac{\beta-\beta_{0}}{\alpha-\alpha_{0}}
		  \in(-\infty,-c_{m_1})\cup(-c_{m_2},\infty)
		\end{displaymath}
		and
		
		\begin{displaymath}
		  \frac{\tilde\beta-\beta_{0}}{\tilde\alpha-\alpha_{0}}
		  \in(-c_{m_1},-c_{m_2}).
		\end{displaymath}
		We now take a neighbourhood $V\subset V_{1}$ of $0$
		such that the Brouwer degrees of $\varphi(\cdot,\alpha,\beta)$
		and $\varphi(\cdot,\tilde\alpha,\tilde\beta)$ on $V$ with respect to $0$
		are the same as the signs of $\det[\varphi'_{\xi}(0,\alpha,\beta)]$
		and $\det[\varphi'_{\xi}(0,\tilde\alpha,\tilde\beta)]$ respectively.
		We get
		
		\begin{displaymath}
		  \deg(\varphi(\cdot,\alpha,\beta),V,0)=
		  \mbox{sgn}\det[\varphi'_{\xi}(0,\alpha,\beta)]=1
		\end{displaymath}
		and
		
		\begin{displaymath}
		  \deg(\varphi(\cdot,\tilde\alpha,\tilde\beta),V,0)=
		  \mbox{sgn}\det[\varphi'_{\xi}(0,\tilde\alpha,\tilde\beta)]=-1,
		\end{displaymath}
		which contradicts the equality \eqref{KrasAlter}.
		Hence $(0,\alpha_{0},\beta_{0},\gamma_{0})$ is a branching point
		of the equation \eqref{OE}.
		    
    
    \bibstyle

    \vspace{1cm}
    
    \noindent
    \textsc{Marek Izydorek}\\
    Faculty of Applied Physics and Mathematics\\
    Gda\'{n}sk University of Technology\\
    Narutowicza 11/12, 80-233 Gda\'{n}sk, Poland\\
    izydorek@mif.pg.gda.pl
    
    \vspace{1cm}
    
    \noindent
    \textsc{Joanna Janczewska}\\
    Faculty of Applied Physics and Mathematics\\
    Gda\'{n}sk University of Technology\\
    Narutowicza 11/12, 80-233 Gda\'{n}sk, Poland\\
    janczewska@mif.pg.gda.pl
    
    \vspace{1cm}
    
    \newpage
    \noindent
    \textsc{Nils Waterstraat}\\
    School of Mathematics, Statistics and Actuarial Science\\
    University of Kent\\
    Canterbury\\
    Kent CT2 7NF\\
    UNITED KINGDOM\\
    N.Waterstraat@kent.ac.uk
    
    \vspace{1cm}
    
    \noindent
    \textsc{Anita Zgorzelska}\\
    Faculty of Applied Physics and Mathematics\\
    Gda\'{n}sk University of Technology\\
    Narutowicza 11/12, 80-233 Gda\'{n}sk, Poland\\
    azgorzelska@mif.pg.gda.pl
      

\begin{thebibliography}{99}
      \bibitem{AmP} A.\ Ambrosetti, G.\ Prodi, \textit{A Primer of Nonlinear
      Analysis}, Cambridge University Press, 1993.      
      \bibitem{BoB1} G.\ Bonanno, B.\ Di Bella, A boundary value problem
      for fourth-order elastic beam equations, \textit{J. Math. Anal. Appl.}
      343 (2008), no.\ 2, 1166--1176.
      \bibitem{BoB2} G.\ Bonanno, B.\ Di Bella, Infinitely many solutions
      for a fourth-order elastic beam equation, \textit{NoDEA Nonlinear
      Differential Equations Appl.} 18 (2011), no.\ 3, 357--368.
      \bibitem{BBO} G.\ Bonanno, B.\ Di Bella, D.\ O'Regan, Nontrivial
      solutions for nonlinear fourth-order elastic beam equations,
      \textit{Comput. Math. Appl.} 62 (2011), no.\ 4, 1862--1869.
      \bibitem{BChT} G.\ Bonanno, A.\ Chinn\`{i}, S.\ Tersian,
      Existence results for a two point boundary value problem
      involving a fourth-order equation, \textit{Electron. J. Qual. Theory
      Differ. Equ.} 2015, no.\ 33, 9 pp.
      \bibitem{BoD} A.\ Borisovich, J.\ Dymkowska, \textit{Elements
      of Functional Analysis with Applications in Elastic Mechanics},
      Gda\'{n}sk University of Technology, Gda\'{n}sk, 2003 (in Polish).
      \bibitem{BMSz} A.\ Borisovich, Yu.\ Morozov, Cz.\ Szymczak,
      Bifurcations of the forms of equilibrium of nonlinear elastic beam
      lying on the elastic foundation, Preprint no.\ 136 (2000), Institute
      of Mathematics, University of Gda\'{n}sk.
      \bibitem{ChL} I.\ Chueshow, I.\ Lasiecka, \textit{Von Karman Evolution
      Equations. Well-posedness and Long-Time Dynamics}, SMM, Springer,
      New York, 2010.
      \bibitem{CwR} A. \'{C}wiszewski, K.\ Rybakowski, Singular dynamics
      of strongly damped beam equation, \textit{J. Differential Equations}
      247 (2009), no.\ 12, 3202--3233.
      \bibitem{Goldberg} I.\ Gohberg, S.\ Goldberg, M.A.\ Kaashoek,
      \textit{Classes of Linear Operators}, Vol.\ I, Oper. Theory Adv. Appl. 49,
      Birkh\"{a}user, Basel, 1990.
      \bibitem{Jan} J.\ Janczewska, Local properties of the solution set
      of the operator equation in Banach spaces in a neighbourhood
      of a bifurcation point, \textit{Cent. Eur. J. Math.} 2 (2004), no.\ 4,
      561--572.
      \bibitem{Red} J.N.\ Reddy, \textit{Energy Principles and Variational
      Methods in Applied Mechanics}, John Wiley and Sons, Inc., Hoboken,
      New Jersey, 2002.
      \bibitem{Sap} Yu.I.\ Sapronov, \textit{Branching of Solutions of Smooth
      Fredholm Equations}, Lect.\ Notes Math. 1108, Springer-Verlag, 1982.
    \end{thebibliography}
  \end{document}